\documentclass[a4paper,12pt]{article}
\usepackage[margin=1in]{geometry}
\usepackage{amsmath} 
\usepackage{amssymb} 
\usepackage{amsthm} 
\usepackage{bbold}
\usepackage{hyperref}
\usepackage[version=0.96]{pgf}
\usepackage{tikz}
\usetikzlibrary{arrows,matrix}

\hypersetup{colorlinks=true}
\hypersetup{linkcolor=black,citecolor=black,urlcolor=black,filecolor=black,menucolor=black}

\newtheorem{theorem}{\bf Theorem}[section]
\newtheorem{definition}[theorem]{\bf Definition}
\newtheorem{lemma}[theorem]{\bf Lemma}

\newtheorem{remark}[theorem]{\bf Remark}
\newtheorem{example}[theorem]{\bf Example}

\DeclareMathOperator{\id}{id}

\DeclareMathOperator{\Ker}{\mathcal{N}}
\DeclareMathOperator{\Ran}{\mathcal{R}}

\newcommand{\A}{\ensuremath{\mathcal{A}}}
\newcommand{\R}{\ensuremath{\mathcal{R}}}

\def\bmatrix{\left[\begin{array}}
\def\endbmatrix{\end{array}\right]}
\def\d{\begin{definition}}
\def\ed{\end{definition}}
\def\t{\begin{theorem}}
\def\et{\end{theorem}}
\def\p{\noindent{\bf Proof.  }}

\title{Algebraic proof methods for identities of matrices and operators: improvements of Hartwig's triple reverse order law}
\author{Dragana S. Cvetkovi\'c-Ili\'c\textsuperscript{1}, Clemens Hofstadler\textsuperscript{2}, Jamal Hossein Poor\textsuperscript{2},\and Jovana Milo\v{s}evi\'c\textsuperscript{1}, Clemens G. Raab\textsuperscript{2}, and Georg Regensburger\textsuperscript{2}}
\date{\small{
\textsuperscript{1}Department of Mathematics, Faculty of Sciences and Mathematics,  \\ University of Ni\v{s}, Serbia\\
\textsuperscript{2}Institute for Algebra, Johannes Kepler University Linz, Austria}}

\begin{document}
\maketitle

\begin{abstract}
When improving results about generalized inverses, the aim often is to do this in the most general setting possible by eliminating superfluous assumptions and by simplifying some of the conditions in statements.
In this paper, we use Hartwig's well-known triple reverse order law as an example for showing how this can be done using a recent framework for algebraic proofs and the software package \texttt{OperatorGB}.
Our improvements of Hartwig's result are proven in rings with involution and we discuss computer-assisted proofs that show these results in other settings based on the framework and a single computation with noncommutative polynomials.
\end{abstract}



\section{Introduction}

Introducing generalized inverses and developing tools working with them in the case when ordinary inverses do not exist, resulted in a lot of progress in several branches of mathematics and many other fields outside of mathematics (mechanics, robotics, control theory, automation, etc.).
The importance and usefulness of this area of research is demonstrated by various open problems that have been solved using the theory of generalized inverses and by many published results.
However, a lot of recently published results for generalized inverses and their applications were proved only under restrictive assumptions which limit their applications to certain very particular cases.
One reason for that is that, in contrast to the setting of matrices, generalized inverses are not defined for each element of more general settings considered (algebras of operators, $C^*$-algebras, rings, \ldots).
In order to benefit from the rich theory of generalized inverses and many already developed useful techniques, researchers usually impose existence of generalized inverses when proving statements.
This leads to many results with redundant instances of assuming regularity of certain elements which makes them less applicable.

The basic example for unnecessary regularity assumptions is the matrix equation $AXB=C$, which was one of the first applications of the later called Moore-Penrose inverse that was introduced by Moore and Penrose independently.
Its solvability and the general solution were considered by Penrose in 1955 \cite{Penrose} in the same paper in which he introduced the four Penrose equations.
Since this result is almost algebraic, it was very easy to generalize it for example to the case of operator equations $AXB=C$ but under the additional assumptions of the closedness of the ranges of the bounded linear operators $A$ and $B$ (that is equivalent with the existence of their Moore-Penrose inverses for operators on Hilbert spaces).
Solvability of this equation in the general case was only considered several years ago, see \cite{Arias}, but many other problems, such as, for example, the existence of a positive solution of that same equation, are still open in the general case.
In fact, there are a lot of problems like this where we have an answer only in some particular cases.
So, in the recent years a lot of effort has been made to widen the range of applicability of these results by considering more general cases of the problems without imposing any additional assumptions.
This paper is exactly one such important step in generalizing Hartwig's triple reverse order law.

In this paper, we present several significant improvements of Hartwig's triple reverse order law motivated by using the software package \texttt{OperatorGB} \cite{HofstadlerRaabRegensburger2019}, which is based on \cite{RaabRegensburgerHosseinPoor2019,Hofstadler2020}.
The aim is to prove statements in an abstract setting in such a way that analogous statements in various concrete settings (e.g.\ for matrices, linear bounded operators, $C^*$-algebras, \ldots) can easily be proven in a rigorous way, but without inspecting every step of the proof of the abstract statement.
To this end, we employ a recent framework that allows to produce rigorous proofs for several different concrete settings by translating a single statement about abstract noncommutative polynomials.
This framework was developed in \cite{RaabRegensburgerHosseinPoor2019} and the software package \texttt{OperatorGB} provides extensive computer support for doing the computations needed.
In particular, the software provides explicit certificates of identities, which can be checked independently.
Moreover, the software can also be used to explore variations of given statements.
That is what initiated the improvements of Hartwig's triple reverse order law presented in this paper.
Based on the results obtained by this software we give a hand proof in the setting of rings which hopefully provides motivation for further research with the same idea.
In addition, we explain how computer-assisted proofs of all these improvements can be done and we provide a \textsc{Mathematica} notebook containing all these automated proofs at
	\url{http://gregensburger.com/softw/OperatorGB}.
These improvements are the first new results that are obtained by applying the framework and software.
From this website also a \textsc{Mathematica} as well as a \textsc{SageMath} version of the \texttt{OperatorGB} package can be obtained.

The main setting that we consider in this paper is a ring  $\R$ with a unit $1\neq 0$ and an involution $a
\mapsto a^*$ satisfying 
$$
 (a^*)^* = a, \quad \quad (a + b)^* = a^* + b^*, \quad\quad (ab)^* = b^*a^*.
$$

\begin{definition}
We say that $a\in\R$ is \emph{Moore-Penrose invertible (or MP-invertible)}, if there exists
$b\in\R$ such that the following hold:
  \begin{equation} \label{eq-more}
aba=a,\ bab=b,\ (ab)^*=ab,\
  (ba)^*=ba.\end{equation}
An element  $b$ that satisfies \eqref{eq-more} is called a {\em Moore-Penrose
inverse\/} of $a$.
\end{definition}

It is well known that the Moore-Penrose inverse is unique
when it exists. We denote the Moore-Penrose inverse of $a$ by $a^\dag$.
We point out
some properties of the Moore-Penrose inverse that follow from the definition.
Clearly,  $a$ is MP-invertible if and only if $a^*$ is MP-invertible; in
this case
\[     (a^*)^\dag = (a^\dag)^*.\]
If $a$ is MP-invertible, then so are $a^*a$ and $aa^*$, with
\[  (a^*a)^\dag = a^\dag (a^*)^\dag,\quad (aa^*)^\dag = (a^*)^\dag a^\dag.\]

\begin{definition}
An element $a\in\R$ is {\em left $*$-cancellable} if, for all $z\in\R$,
$a^*az = 0$ implies $az=0$,  it is {\em right
$*$-cancellable} if, for all $z\in\R$, $zaa^*=0$ implies $za=0$, and {\em
$*$-cancellable} if it is both left and right cancellable.
\end{definition}

We observe that
$a$ is left $*$-cancellable if and only if $a^*$ is right $*$-cancellable.
In a $C^*$-algebra, every element is $*$-cancellable:  If
$a^*az=0$, then $(az)^*az= 0$ which implies $az=0$; similarly $zaa^* = 0$ implies $za = 0$.

If $b\in\R$ satisfies $\{i,\ldots,j\}$ of the Penrose equations from $(\ref{eq-more})$ we say that $b$ is a $\{i,\ldots,j\}$-inverse of $a$. The set of all $\{i,\ldots,j\}$-inverses of $a$ is denoted by $a\{i,\ldots,j\}$. Evidently $a\{1,2,3,4\}=\{a^\dagger\}$.  We say that an element $a\in\R$ is regular if $a\{1\}\neq\emptyset$. In general, in $C^*$-algebras we have that the regularity property is equivalent with MP-invertibility. In particular, in an algebra of bounded linear operators the regularity of an arbitrary operator $A$ is equivalent to the closedness of the range of $A$ while in a ring with involution MP-invertibility of $m$ is equivalent to the right $*$-cancellability of $m$ and group invertibility of $mm^*$ (see Theorem 8.25 from \cite{Rao} or Theorem 5.3 from \cite{KP}).
\begin{definition}
An element $a\in\R$ is \emph{EP} if $a\R=a^*\R$.
\end{definition}

In the following subsection, we give a self-contained informal overview of the framework for algebraic proofs and of the software package \texttt{OperatorGB}.
In Section~\ref{sec:Hartwig}, we first discuss Hartwig's triple reverse order law and related results from the literature. Then, we give hand proofs of several improvements of it in rings with involution.
After that, in Section~\ref{sec:computerproof}, we discuss how these results can be proven with the help of the computer in such a way that the framework yields rigorous proofs for these statements also in the context of matrices and operators. 
Formal definitions and statements about the framework for algebraic proofs, which is used by the software \texttt{OperatorGB}, are summarized in the appendix.

\subsection{Introduction to the framework for algebraic proofs}\label{sec:introductionframework}

The advantage of the framework presented below is that a single computation in an abstract setting proves analogous statements in various concrete settings (e.g.\ for matrices, linear bounded operators, $C^*$-algebras, \ldots) without having to inspect every step of the abstract computation.
Just like in any ring, computations with noncommutative polynomials allow any two elements to be added or multiplied.
Therefore, it is not clear a priori that a given proof of a statement in a ring is valid also for rectangular matrices or operators with domains and codomains.
Using the framework for algebraic proofs, the following steps have to be carried out once in a suitable ring of noncommutative polynomials.
Then, to rigorously prove a statement for various concrete settings, based on Theorem~\ref{thm:RlinearUnique}, it suffices to check that the polynomials corresponding to the assumptions and claims are compatible with different domains and codomains of operators.
\begin{enumerate}
 \item Express all assumptions and claimed properties as identities in terms of operators.
 \item Take the differences of the left and right hand sides of these identities and replace the individual operators uniformly by noncommutative indeterminates in order to convert the identities into polynomials.
 \item Find a concrete representation of the polynomials corresponding to the claim as a two-sided linear combination of polynomials corresponding to the assumptions, where coefficients are polynomials.
\end{enumerate}
Representations of polynomials as mentioned in the last step are called \emph{cofactor representations} and serve as certificates for ideal membership that can be checked independently of how they were found. 
However, finding them is a hard problem, since for noncommutative polynomials ideal membership is undecidable in general, see e.g.\ \cite{Mora1994}.
In practice, cofactor representations often can be found by computing a (partial) Gr\"obner basis, see \cite{Hofstadler2020} and references therein.
Already in the pioneering work \cite{HeltonWavrik1994,HeltonStankusWavrik1998} Gr\"obner bases have been used to simplify matrix identities in linear systems theory.
Proving operator identities using Gr\"obner basis computations and related questions are also addressed in \cite{LevandovskyySchmitz2020}.

The software package \texttt{OperatorGB} provides the command \texttt{Certify}, which not only tries to compute cofactor representations but also does the compatibility checks of assumptions and claims.
Inspecting the explicit cofactor representations found by the software can also give hints how assumptions could be relaxed by dropping the assumptions that do not appear in the cofactor representations.
More generally, the software makes it easy to experiment with different sets of assumptions for proving a desired claim.
Improvements of Hartwig's triple reverse order law found by such experiments were the basis for the results presented in the next section.
For details on how our framework and software are used to find and prove these results, see Section~\ref{sec:computerproof}.

Next, we illustrate the approach with a simple statement about inner inverses of matrices, for details of the framework see the appendix.
In \cite[Thm.~2.3]{Werner1994}, Werner proved among other things the following statement about inner inverses of complex matrices. If $A$ and $B$ are complex matrices such that $AB$ exists, then 
$\Ker(A)\subseteq\Ran(B)$ implies that $B\{1\}A\{1\} \subseteq (AB)\{1\}$.
As a first step, we have to phrase all properties stated in the assumptions and in the claim in terms of identities of matrices, which results in the following statement. For any complex matrices $A^-,B^-$ with
\begin{equation}
\label{eq:WernerInnerInverses}
 AA^-A=A \quad\text{and}\quad BB^-B=B,
\end{equation}
we have that
\begin{equation}
\label{eq:WernerAssumption}
 BB^-(I-A^-A)=I-A^-A
\end{equation}
implies
\begin{equation}
\label{eq:WernerClaim}
 ABB^-A^-AB=AB.
\end{equation}
The formats of these matrices can be visualized by the following diagram.
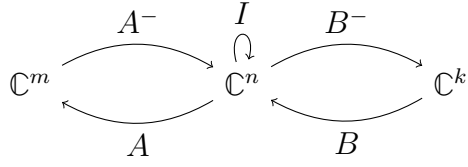
\begin{center}
  \begin{tikzpicture}
    \matrix (m) [matrix of math nodes, column sep=2cm]
     {\mathbb{C}^m & \mathbb{C}^n & \mathbb{C}^k\\};
    \path[->] (m-1-1) edge [bend left] node [auto] {$A^{-}$} (m-1-2);
    \path[->] (m-1-2) edge [bend left] node [auto] {$A$} (m-1-1);
    \path[->] (m-1-2) edge [loop above] node [auto] {$I$} (m-1-2);
    \path[->] (m-1-2) edge [bend left] node [auto] {$B^{-}$} (m-1-3);
    \path[->] (m-1-3) edge [bend left] node [auto] {$B$} (m-1-2);
  \end{tikzpicture}
\end{center}
Secondly, we represent these identities by noncommutative polynomials in the indeterminates $\{a,a^-,b,b^-,i\}$. This is done by uniformly replacing each matrix (including the identity matrix) by an indeterminate and forming the difference of the left and right hand side of each identity.
\begin{gather}
f_1 = aa^-a-a \qquad f_2 = bb^-b-b \qquad f_3 = bb^-(i-a^-a)-i+a^-a\\
f = abb^-a^-ab-ab\label{eq:WernerClaimPoly}
\end{gather}
Moreover, for correctly handling the identity matrix, we also need to represent its algebraic identities in terms of polynomials.
\begin{equation}
f_4 = ai-a \quad f_5 = ia^--a^- \quad f_6 = ib-b \quad f_7 = b^-i-b^- \quad f_8 = i^2-i
\end{equation}
Finally, either by hand or with the help of software, we can express the polynomial $f$ representing the claim in terms of the polynomials $f_1,\dots,f_8$ representing the assumptions.
\begin{align}
 f &= f_1b+af_2-af_3b+(abb^--a)f_6\label{eq:WernerCofactor1}
\end{align}
By Theorem~\ref{thm:RlinearUnique}, it follows from \eqref{eq:WernerCofactor1} that \eqref{eq:WernerClaim} holds for any matrices $A,B$ with inner inverses $A^-,B^-$ satisfying \eqref{eq:WernerAssumption}, see Lemma~\ref{lem:Werner1} in the appendix.
Moreover, based on the theorem, the cofactor representation \eqref{eq:WernerCofactor1} also proves the analogous statement for bounded linear operators $A,B$ between Hilbert spaces $U,V,W$ as in the following diagram.
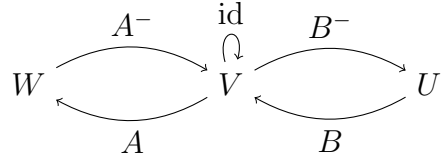
\begin{center}
\begin{tikzpicture}
 \matrix (m) [matrix of math nodes, column sep=2cm]
  {W & V & U\\};
 \path[->] (m-1-1) edge [bend left] node [auto] {$A^{-}$} (m-1-2);
 \path[->] (m-1-2) edge [bend left] node [auto] {$A$} (m-1-1);
 \path[->] (m-1-2) edge [loop above] node [auto] {$\id$} (m-1-2);
 \path[->] (m-1-2) edge [bend left] node [auto] {$B^{-}$} (m-1-3);
 \path[->] (m-1-3) edge [bend left] node [auto] {$B$} (m-1-2);
\end{tikzpicture}
\end{center}

As mentioned above, explicit cofactor representations not only certify ideal membership, but can also give hints how assumptions could be relaxed.
In particular, they also allow to analyze which assumptions can be relaxed for proving a given identity of operators. For example, \eqref{eq:WernerCofactor1} does not involve $f_4,f_5,f_7,f_8$, so in $BB^-(I-A^-A)=I-A^-A$ we could replace the identity matrix $I$ by any other matrix $J$ satisfying $JB=B$.
Trivially, any cofactor representation with polynomials having only integer coefficients, as in \eqref{eq:WernerCofactor1} above, also holds in any ring, and hence proves an analogous statement for rings.

As discussed before, to apply the proof framework directly, one has to translate all properties of the operators involved into identities.
In the context of generalized inverses, such properties are often conditions on ranges and kernels of some basic operators.
If a projection (idempotent) on these spaces can be expressed in terms of basic operators, the translation to identities is immediate, as illustrated in the example above.
Inclusion of ranges $\Ran(A)\subseteq \Ran(B)$ can be translated in many situations to the existence of a factorization $A=BC$ for some operator $C$.
In Hilbert or Banach spaces, this is the well-known factorization property in Douglas' lemma. For proving the existence of such a linear operator $C$ without any additional properties, one just needs operators defined on a vector space over an arbitrary field.
This principle will play a prominent role in Section~\ref{sec:computerproof}.

\section{Improvements of Hartwig's triple reverse order law}
\label{sec:Hartwig}

The {\lq\lq}reverse order law{\rq\rq} problem was originally posed by
Greville \cite{G} as early as in the $1960$'s, who first considered it in
the case of the Moore-Penrose inverse of the product of two matrices. Namely, for
given matrices $A,B$ such that $AB$ is defined  the following was
proved:
\begin{eqnarray}\label{G1}
(AB)^\dagger=B^\dagger A^\dagger \Leftrightarrow \R(A^*AB)\subseteq \R(B),\ \R(BB^*A^*)\subseteq\R(A^*).
\end{eqnarray}

This was followed by further research on this subject branching in several directions:
\begin{itemize}
\item[-] for products of more than two matrices,
\item[-] for different classes of generalized inverses ($\{1\}$, $\{1,3\}$, $\{1,2,3\}$, etc.), and
\item[-] in different settings (operator algebras, $C^*$-algebras, rings, etc.).
\end{itemize}

\smallskip

For more information on this subject please see
\cite{B4,DW,1,2,3,4,5,6,7,8, PW, P1, LY,LY1, Liu2, J1, SSS,T,
Wang2}.

\smallskip

One of the first to be inspired by Greville's result $(\ref{G1})$ was Hartwig \cite{H}, who studied the reverse order law for the Moore-Penrose inverse of the product of three matrices. Indeed, he considered necessary and sufficient conditions such that
\begin{eqnarray}\label{e1}
(ABC)^{\dagger}=C^\dagger B^\dagger A^\dagger
\end{eqnarray}
holds.

\t $\cite{H}$ \label{t0} Let $A,B,C$ be complex matrices such that $ABC$ is defined and let $ P=A^\dag ABCC^\dag$, $Q=CC^\dag B^\dag A^\dag A$. The following conditions are
equivalent:
\begin{itemize}
\item [$(i)$]  $(ABC)^{\dagger}=C^\dagger B^\dagger A^\dagger;$

\item [$(ii)$]  $Q\in P\{1,2\}$ and both of $A^*APQ$ and $QPCC^*$ are
Hermitian;

\item [$(iii)$]  $Q\in P\{1,2\}$ and both of $A^*APQ$ and $QPCC^*$ are
EP;

\item [$(iv)$]  $Q\in P\{1\},$ $\R(A^*AP)=\R(Q^*)$ and $\R(CC^*P^*)=\R(Q)$;

\item [$(v)$]  $PQ=(PQ)^2,$ $\R(A^*AP)=\R(Q^*)$ and $\R(CC^*P^*)=\R(Q)$.
\end{itemize}
\et

This inspired many authors to continue research in these directions and it was precisely Hartwig's result that motivated further consideration of the reverse order law for MP-inverses in the case of three elements in certain other settings such as in the algebra of bonded linear operators and in $C^*$-algebras, which was done in \cite{ND} and \cite{JM}, respectively.
In both papers, results analogous to Hartwig's paper were obtained, but with the additional conditions of  regularity of all three elements and their products.
Here, we mention a result presented in \cite{JM} for the case of $C^*$-algebras in order to give a clear picture of the conditions assumed and the equivalences obtained (in the case of bounded linear operators between Hilbert spaces the theorem looks identically).

  \t \label{t1} $\cite{JM}$ Let $\A$ be a complex unital $C^*$-algebra and let $a,b,c\in\A$ be such that
$a,b,c$ and $abc$ are regular. Let $p=a^\dagger abcc^\dagger$ and $q=cc^\dagger b^\dagger a^\dagger a$. 
Then, the following conditions are
equivalent:
\begin{itemize}
\item [$(i)$]  $(abc)^{\dagger}=c^\dagger b^\dagger a^\dagger;$

\item [$(ii)$]  $q\in p\{1,2\}$ and both of $a^*apq$ and $qpcc^*$ are
Hermitian;

\item [$(iii)$]  $q\in p\{1,2\}$ and both of $a^*apq$ and $qpcc^*$ are
EP;

\item [$(iv)$]  $q\in p\{1\},$ $a^*ap\A=q^*\A$ and $cc^*p^*\A=q\A$;

\item [$(v)$]  $pq=(pq)^2,$ $a^*ap\A=q^*\A$ and $cc^*p^*\A=q\A$.
\end{itemize}
\et

The main results presented here  represent an important improvement of Hartwig's result in several senses:

\begin{itemize}

\item[$\circ$] We consider the problem in rings with involution, which is a more abstract setting than what was considered in the literature so far.
Together with the framework and the discussion in Section~\ref{sec:computerproof} this generalizes all the results previously mentioned.

\item[$\circ$] We relax conditions
$(iv)$ and $(v)$ in the original result of Hartwig (Theorem \ref{t0}),  by replacing the respective equalities of ranges assumed in both of these conditions with appropriate inclusions of ranges.
For example, we show in Theorems~\ref{t3} and \ref{t4} that certain combinations of inclusions (there are four of them in total), along with the assumption that the element $pq$ is idempotent, imply $(\ref{e1})$, while the other two combinations do not guarantee the claimed conclusion (see Example~\ref{ex:Hartwig5inclusions}).
As for the analogous results for algebras of operators and $C^*$-algebras (see \cite{ND} and \cite{JM}), we improve them in a similar way by replacing equalities with appropriate inclusions.

\item[$\circ$] Compared to the results for algebras of operators and $C^*$-algebras in general (see \cite{ND} and \cite{JM}), we significantly reduce the set of starting assumptions upon which these results are based by dropping certain regularity conditions. Namely, if one is interested in the validity of $(\ref{e1})$, it is possible to omit the requirement that  the product $abc$ is MP-invertible, since this follows directly from some of the assumptions $(iv)$ or $(v)$. In the case of rings, MP-invertibility of the product $abc$ can be replaced with the weaker condition of right $*$-cancellability of $abc$. See Theorems~\ref{t3} and \ref{tnova} and similarly Theorem~\ref{t4}.

\item[$\circ$]  Also, it is possible to generalize the result by showing that $b^\dagger$ can be replaced by an arbitrary element $\widetilde{b}$ as well as that $a^\dagger$ and $c^\dagger$ can be replaced with arbitrary $a^{(1,2,3)}$ and  $c^{(1,2,4)}$, respectively  (see Theorem \ref{t5}).
In this way, the assumption of MP-invertibility of the element $b$  is dropped and the MP-invertibility of the elements $a$ and $c$ is replaced with the existence of $a^{(1,2,3)}$ and  $c^{(1,2,4)}$. This, although the last two are equivalent conditions in  operator algebras and $C^*$-algebras, improves the results significantly in rings with involution since there the existence of a $\{1,2,3\}$-inverse of an element is equivalent with the existence of its $\{1,3\}$-inverse and the latter is a much weaker condition than MP-invertibility (as witnessed by the ring $M_2(\Bbb C)$ with taking transposes as the involution).

\end{itemize}

Recall that $\R$ denotes a ring with a unit $1\neq 0$ and with an involution.

\t \label{t3} Let $a,b,c\in\R$ be such that
 $a,c$ are MP-invertible. Let  $p=a^\dagger abcc^\dagger$ and $q=cc^\dagger \widetilde{b} a^\dagger a$,  for $\widetilde{b}\in\R$.  Then, the following conditions are equivalent:
\begin{itemize}
\item [$(i)$]  $abc$ is Moore-Penrose invertible and $(abc)^{\dagger}=c^\dagger \widetilde{b}a^\dagger$;

\item [$(iv)$]  $q\in p\{1\},$ $a^*ap\R\supseteq q^*\R$ and $cc^*p^*\R\subseteq q\R$;

\item [$(v)$]     $abc$ is right $*$-cancellable, $pq=(pq)^2,$ $a^*ap\R\supseteq q^*\R$ and $cc^*p^*\R\subseteq q\R$;

\item [$(vi)$]  $q\in p\{2\}$,  $a^*ap\R\supseteq q^*\R$ and $cc^*p^*\R\subseteq q\R$.
\end{itemize}
\et

 \p Let $m=abc$ and $ \widetilde{m}=c^\dagger \widetilde{b}
a^\dagger$.  Evidently, $pq$ is idempotent if and only if $m\widetilde{m}$ is idempotent. Also, we have that the following equivalences hold:

\begin{eqnarray*}
&&a^*ap\R\supseteq q^*\R\Leftrightarrow m\R\supseteq(\widetilde{m})^*\R \Leftrightarrow \R m^*\supseteq\R \widetilde{m} \Leftrightarrow \widetilde{m}\in\R m^*;\\
&&cc^*p^*\R\subseteq q\R\Leftrightarrow m^*\R\subseteq \widetilde{m}\R \Leftrightarrow m^*\in \widetilde{m}\R;
\end{eqnarray*}

$(i)\Rightarrow(v)$:  If
$m^{\dagger}=\widetilde{m}$, then clearly $m\widetilde{m}$ is
idempotent. Also,
\begin{eqnarray*}
&&\widetilde{m}=m^\dagger=m^\dagger mm^\dagger=m^\dagger
(m^\dagger)^*m^*\in\R m^*,\\
&&m^*=(mm^\dagger m)^*=m^\dagger mm^*=\widetilde{m} mm^*\in \widetilde{m}\R.
\end{eqnarray*}

$(v)\Rightarrow(i)$: If $(v)$ holds, then there exist $u,v\in\R$ such that
$\widetilde{m}=um^*$ and $m^*=\widetilde{m} v$.
Now, multiplying
$m\widetilde{m}=(m\widetilde{m})^2$  by $v$
from the right side, we get $mm^*=m\widetilde{m} mm^*$ i.e.\
$(1-m\widetilde{m})mm^*=0$, which gives $(1-m\widetilde{m})m=0$ by
right $*$-cancellability of $m$.
So, $\widetilde{m}$ is an inner
inverse of $m$.
Further, we have that
$$
\widetilde{m}=um^*=u(m\widetilde{m}
m)^*=\widetilde{m}(m\widetilde{m})^*,
$$
which implies that $m\widetilde{m}$ is Hermitian and further
$$
\widetilde{m}=\widetilde{m}(m\widetilde{m})^*=\widetilde{m}
m\widetilde{m} .
$$
Also,
$$
m=v^*(\widetilde{m})^*=v^*(\widetilde{m} m\widetilde{m})^*=
m(\widetilde{m} m)^*,
$$
which implies that $\widetilde{m} m$ is Hermitian.

$(iv),(vi)\Rightarrow(v)$: This is evident.

$(i)\Rightarrow(iv)$:  The property $q\in p\{1\}$ follows directly from the fact that $\widetilde{m}$ is an inner inverse of $m$. The rest of the proof follows as in the part $(i)\Rightarrow(v)$.

$(i)\Rightarrow(vi)$:  The property $q\in p\{2\}$ follows from the fact that $\widetilde{m}$ is an outer inverse of $m$. The rest of the proof follows as in the part $(i)\Rightarrow(v)$. $\Box$

\medskip

It is interesting to mention that if we take the reverse inclusion from $(ii)$ of Theorem \ref{t3} (notice that in Hartwig's result we have equality!) and replace in the statement of the theorem the assumption  of  right $*$-cancellability of  $abc$ with the  assumption of left $*$-cancellability of  $c^\dagger \widetilde{b} a^\dagger$, we get the following analogous result.

\t \label{t4} Let  $a,b,c,\widetilde{b}\in\R$ be such that
 $a,c$ are MP-invertible. Let  $p=a^\dagger abcc^\dagger$ and $q=cc^\dagger \widetilde{b} a^\dagger a$. Then, the following conditions
are equivalent:
\begin{itemize}
\item [$(i)$]  $abc$ is Moore-Penrose invertible and $(abc)^{\dagger}=c^{\dagger} \widetilde{b}a^{\dagger}$;

\item [$(iv)$]  $q\in p\{1\}$, $a^*ap\R\subseteq q^*\R$ and $cc^*p^*\R\supseteq q\R$;

\item [$(v)$]  $c^\dagger  \widetilde{b}a^\dagger$ is   left $*$-cancellable, $pq=(pq)^2,$ $a^*ap\R\subseteq q^*\R$ and $cc^*p^*\R\supseteq q\R$;

\item [$(vi)$]  $q\in p\{2\}$, $a^*ap\R\subseteq q^*\R$ and $cc^*p^*\R\supseteq q\R$.

\end{itemize}
\et

The following example illustrates the fact that the remaining two combinations of inclusions in the original result of Hartwig (Theorem \ref{t0} $(v)$) do not necessarily imply $(\ref{e1})$.

\begin{example}
\label{ex:Hartwig5inclusions}
Let
$$
A=\bmatrix{ccc} -3&2&2\\0&0&0\\0&0&0\endbmatrix,\ \ B=\bmatrix{ccc} 1&0&1\\0&1&1\\1&0&0\endbmatrix,\ \  C=\frac 1 3 \bmatrix{ccc} 1&1&1\\1&1&1\\1&1&1\endbmatrix.
$$
Then
$$
A^\dagger=\frac 1 {17}\bmatrix{ccc} -3&0&0\\2&0&0\\2&0&0\endbmatrix,\ \ B^\dagger=\bmatrix{ccc} 0&0&1\\-1&1&1\\1&0&-1\endbmatrix,\ \  C^\dagger=C.
$$
If we define $P$ and $Q$ as in Theorem \ref{t0},  we get that $PQ=0$ is idempotent and $\R(A^*AP)\subseteq \R(Q^*)$ and $\R(CC^*P^*)\subseteq\R(Q)$ but $(ABC)^\dagger\neq C^\dagger B^\dagger A^\dagger$.

If matrices $A,B,C$ are defined as $C^\dagger, B^\dagger$ and $A^\dagger$, respectively, as given above, we conclude that also the second pair of inclusions $\R(Q^*)\subseteq \R(A^*AP)$ and $\R(Q)\subseteq \R(CC^*P^*)$ together with the assumption that the matrix $PQ$ is idempotent fails to imply $(\ref{e1})$.

\end{example}
\medskip

On the other hand, the above mentioned pairs of inclusions imply  $(\ref{e1})$  with some assumptions on $p$ and $q$.

\t \label{tnova} Let $a,b,c\in\R$ be such that $a,c$ are
MP-invertible. Let  $p=a^\dagger abcc^\dagger$ and $q=cc^\dagger
\widetilde{b} a^\dagger a$,  for $\widetilde{b}\in\R$.  Then, the
following conditions are equivalent:
\begin{itemize}
\item [$(i)$]  $abc$ is Moore-Penrose invertible and $(abc)^{\dagger}=c^\dagger \widetilde{b}
a^\dagger$;

\item [$(iv)$]  $q\in p\{1\},$ $a^*ap\R\supseteq q^*\R$ and $cc^*p^*\R\supseteq q\R$;

\item [$(vi)$]  $q\in p\{2\}$,  $a^*ap\R\subseteq q^*\R$ and $cc^*p^*\R\subseteq q\R$.
\end{itemize}
\et

In addition to the previously mentioned results, we can show that
MP-invertibility of the elements $a$ and $c$ can be  replaced with
the existence of $a^{(1,2,3)}$ and  $c^{(1,2,4)}$.

\t \label{t5} Let  $a,b,c,\widetilde{b}\in\R$ be such that there exist
$a^{(1,3)}$ and $c^{(1,4)}$ and such that $abc$ is right $*$-cancellable.
Let $a^{(1,2,3)}, c^{(1,2,4)}$ be given such that $c^{(1,2,4)}
\widetilde{b} a^{(1,2,3)}$ is left \mbox{$*$-cancellable} and let
$p=a^{(1,2,3)}abcc^{(1,2,4)}$ and $q=cc^{(1,2,4)} \widetilde{b}
a^{(1,2,3)}a$. Then, the following conditions are equivalent:
\begin{itemize}
\item [$(i)$]  $abc$ is Moore-Penrose invertible and $(abc)^{\dagger}=c^{(1,2,4)} \widetilde{b}
a^{(1,2,3)}$;

\item [$(ii)$]  $q\in p\{1,2\}$ and both of $a^*apq$ and $qpcc^*$ are
Hermitian;

\item [$(iii)$]  $q\in p\{1,2\}$ and both of $a^*apq$ and $qpcc^*$ are EP;

\item [$(iv)$]  $pq=(pq)^2,$ $a^*ap\R\supseteq q^*\R$ and $cc^*p^*\R\subseteq q\R$;

\item [$(v)$]  $pq=(pq)^2,$ $a^*ap\R\subseteq q^*\R$ and $cc^*p^*\R\supseteq q\R$.
\end{itemize}

\et

Notice that, if in Theorem \ref{t5} we replace $a^{(1,2,3)}$ and $c^{(1,2,4)}$ with $a^{(1,3)}$ and $c^{(1,4)}$, respectively, the assertion of the theorem does not hold anymore, which will be shown in the next example:

\begin{example} Let  $B=C=\widetilde{B}=I$ and take any matrix $A$ such that $A\{1,3,4\}\neq \{A^\dagger\}$ (such $A$ can be any projection different from the identity). If we take $A^{(1,3)}=A^{(1,3,4)}\neq A^\dagger$ we get that the conditions $(ii)-(v)$  are all satisfied while $(i)$ from Theorem \ref{t5} is not satisfied.
\end{example}

Finally, by the discussion above we end this section with
the improved version of Hartwig's original result for matrices.

\t  \label{t00} Let $A,B,C$ be complex matrices such that $ABC$ is defined and let $ P=A^\dag ABCC^\dag$, $Q=CC^\dag B^\dag A^\dag A$. The following conditions are
equivalent:
\begin{itemize}

\item [$(i)$]  $(ABC)^{\dagger}=C^\dagger B^\dagger A^\dagger;$

\item [$(ii)$]  $Q\in P\{1,2\}$ and both of $A^*APQ$ and $QPCC^*$ are
Hermitian;

\item [$(iii)$]  $Q\in P\{1,2\}$ and both of $A^*APQ$ and $QPCC^*$ are
EP;

\item [$(iv')$]  $Q\in P\{1\}$, $\R(Q^*)\subseteq \R(A^*AP)$ and $\R(CC^*P^*)\subseteq \R(Q)$;

\item [$(iv'')$]  $Q\in P\{1\}$,
 $\R(A^*AP)\subseteq \R(Q^*)$ and $\R(Q)\subseteq \R(CC^*P^*)$;
 
\item [$(iv''')$]  $Q\in P\{1\}$,
 $\R(Q^*)\subseteq \R(A^*AP)$ and $\R(Q)\subseteq \R(CC^*P^*)$;
 
 \item [$(iv'''')$]  $Q\in P\{2\}$, $\R(Q^*)\subseteq \R(A^*AP)$ and $\R(CC^*P^*)\subseteq \R(Q)$;
 
\item [$(iv''''')$]  $Q\in P\{2\},$
 $\R(A^*AP)\subseteq \R(Q^*)$ and $\R(Q)\subseteq \R(CC^*P^*)$;
 
 \item [$(iv'''''')$]  $Q\in P\{2\},$
 $\R(A^*AP)\subseteq \R(Q^*)$ and $\R(CC^*P^*)\subseteq \R(Q)$;
 
\item [$(v')$]  $PQ=(PQ)^2,$ $\R(Q^*)\subseteq \R(A^*AP)$ and $\R(CC^*P^*)\subseteq \R(Q)$;

\item [$(v'')$]  $PQ=(PQ)^2,$ $\R(A^*AP)\subseteq \R(Q^*)$ and $\R(Q)\subseteq \R(CC^*P^*)$.
\end{itemize}
\et

\subsection{Computer-assisted algebraic proofs}
\label{sec:computerproof}

In the following, we discuss different aspects and use cases of the proof framework outlined in Section~\ref{sec:introductionframework}.
We use Hartwig's result and its improvements presented above to exemplify this.
Algebraically, the central point of the proof is membership of the polynomial representing the claimed identity in the ideal generated by the polynomials representing the assumed identities, c.f.\ the third step listed in the introduction.
Below, we also describe how certain assumptions, which are not identities of matrices or operators themselves, can sometimes still be used within the framework.

First, we focus on the implication $(v)\Rightarrow(i)$ in Theorem~\ref{t0}: if $PQPQ=PQ$, $\Ran(A^*AP)=\Ran(Q^*)$, and $\Ran(CC^*P^*)=\Ran(Q)$, then $M^\dag=C^{\dag}B^{\dag}A^\dag$.

Based on Douglas' lemma, we first translate the range conditions to identities of operators. The four inclusions of ranges are equivalent to the following identities for some operators $U_1,U_2,V_1,V_2$.
\begin{equation}
\label{eq:HartwigDouglasFull}
 A^*AP = Q^*V_1 \quad\quad A^*APV_2 = Q^* \quad\quad
 CC^*P^* = QU_1 \quad\quad CC^*P^*U_2 = Q
\end{equation}
For each Moore-Penrose inverse $A^{\dag},B^{\dag},C^{\dag},M^{\dag}$, we have the four defining identities.

Translating these identities into polynomials, we introduce an indeterminate for each basic operator. Moreover, for each indeterminate, we introduce another indeterminate representing the adjoint of the corresponding operator. In total, this amounts to $22$ indeterminates. Similarly, each identity of operators is translated into two polynomials, one for the identity itself and one for its adjoint. Thereby, we obtain a set $F$ of $34$ noncommutative polynomials with integer coefficients representing the assumptions. The claim corresponds to the polynomial $f=m^{\dag}-c^{\dag}b^{\dag}a^{\dag}$.

Then, we use our software to show that $f$ lies in the ideal generated by the polynomials of $F$. The cofactor representation certifying this ideal membership was computed in less than $45$ seconds and has $157$ terms.
The diagram induced by generic domains and codomains of operators has $4$ vertices and one edge for each indeterminate. 
By construction, the polynomial $f$ and the elements of $F$ are compatible with domains and codomains.
By Theorem~\ref{thm:RlinearUnique}, this now rigorously proves that $M^\dag=C^{\dag}B^{\dag}A^\dag$ holds under the conditions given in $(v)$.
Note that this proof only relies on the defining identities of Moore-Penrose inverses and does not use any additional properties or lemmas.
Consequently, the implication $(v)\Rightarrow(i)$ is in fact proven for any setting in which it can be formulated, since the polynomials in the cofactor representation obtained have only integer coefficients.

Using the software, it is easy to experiment with relaxing the assumptions and check if a cofactor representation of $f$ in terms of a subset of $F$ still can be found. For instance, it turns out that the first and last identity in \eqref{eq:HartwigDouglasFull} can be dropped. This corresponds to relaxing the range conditions in $(v)$ to $\Ran(A^*AP)\supseteq\Ran(Q^*)$ and $\Ran(CC^*P^*)\subseteq\Ran(Q)$. Additionally, we could also observe that the cofactor representation of $f$ contains no polynomial 
associated to any of the four defining equations of $B^\dag$. This shows that $B^\dag$ can in fact be replaced by an arbitrary operator $\tilde{B}$ that does not have to be related to $B$ in any way. 

It is also possible to prove the implication $(i) \Rightarrow (v)$ using our framework and software. To this end, first explicit expressions for $U_1,U_2,V_1,V_2$ in terms of the other basic operators have to be found.
By inspecting the proof of Theorem~\ref{t3} one can see that these can be chosen as
\begin{equation}\label{eq:HartwigDouglasExpressions}
  \begin{aligned}
	U_1 &= BCC^*B^*A^*(A^\dag)^*,\qquad & U_2 &= (B^\dag)^*(C^\dag)^*C^\dag B^\dag A^\dag A,\\
	 V_1 &= B^*A^*ABCC^\dag, & V_2 &= B^\dag A^\dag (A^\dag)^* (B^\dag)^* (C^\dag)^* C^*.
  \end{aligned}
\end{equation}
Then, using the defining equations of $A^{\dag},B^{\dag},C^{\dag},M^{\dag}$, the identity $M^\dag = C^\dag B^\dag A^\dag$ and their adjoint statements as assumptions, the software finds cofactor representations of the polynomial corresponding to
 $PQPQ=PQ$ as well as of the polynomials associated to the four identities in (\ref{eq:HartwigDouglasFull}), where $U_1,U_2,V_1,V_2$ have been replaced by the expressions in (\ref{eq:HartwigDouglasExpressions}).
We note that these cofactor representations only contain polynomials with integer coefficients.
Hence, based on Theorem~\ref{thm:RlinearUnique}, this proves the implication $(i)\Rightarrow(v)$ for any setting in which it can be formulated. 

It is also possible to incorporate properties of operators into this framework that cannot be expressed in terms of identities but only in form of quasi-identities.
In general, quasi-identities are implications where a conjunction of identities implies another identity. 
One example of such a property is $*$-cancellability.
To use these properties to prove a claimed identity, first a suitable polynomial in the ideal representing the assumptions has to be found that corresponds to an operator identity to which such a property is applicable.
Finding such a suitable polynomial is usually a non-trivial task and often has to be done by hand.
For the automated proofs of some of the results presented here, for example, we obtained the required expressions by inspecting the corresponding hand proofs, which were done partly before the automated proofs.
Once such a polynomial has been found, the corresponding quasi-identity can be applied to obtain a new polynomial that corresponds to a shorter identity and that is typically not contained in the ideal that is generated by the polynomials representing the assumptions.
By including this new polynomial into the set of polynomials representing the assumptions, we can enlarge the ideal of all consequences of the assumptions and proceed to prove the ideal membership of the polynomial corresponding to the claimed identity in this larger ideal.

To prove a quasi-identity, the left-hand side of the implication has to be included in the assumptions and the right-hand side becomes the claimed identity.
When translating these operator identities into polynomials it is important to introduce new indeterminates that do not satisfy any additional identities for all universally quantified operators in the quasi-identity.
Then, to prove the quasi-identity, it only remains to prove the ideal membership of the polynomial associated to the claim in the ideal generated by the polynomials representing the assumptions.

Based on the discussion and the observations made above, it is no surprise that the software can also be used to prove all the improved results of Hartwig's triple reverse order law presented in this work.
In the following, we explain how this can be done using the equivalence $(i) \Leftrightarrow (v)$ of Theorem~\ref{t3}.

For the implication $(v)\Rightarrow (i)$, we translate the assumptions $pq=(pq)^2,$ $a^*ap\R\supseteq q^*\R$, $cc^*p^*\R\subseteq q\R$ and their adjoint statements into polynomials.
Note that in order to translate the set inclusions we can use factorizations analogous to \eqref{eq:HartwigDouglasFull}.
In contrast to the original statement of Hartwig, where the MP-invertibility of $ABC$ is already given, we now have to prove that $m = abc$ is MP-invertible and that $m^\dagger = c^\dagger \tilde{b} a^\dagger$.
Hence, the claim is that $\tilde{m} = c^\dagger \tilde{b} a^\dagger$ satisfies the four defining equations of $m^\dagger$. However, trying to show the ideal membership of the corresponding polynomials in the ideal generated by the polynomials representing the assumptions fails.
This is because these polynomials do not contain any information about the right $*$-cancellability of $m$. To use this property, we have to find a polynomial in the ideal generated by the polynomials associated to our assumptions that corresponds to an identity to which this property is applicable.
In the hand proof of this implication, the right $*$-cancellability is applied to $(1- m\tilde{m})mm^* = 0$.
Using the software, we can show that the polynomial corresponding to this identity is indeed contained in the ideal generated by the polynomials representing the assumptions.
Hence, as in the hand proof, we can apply the right $*$-cancellability of $m$ to $(1- m\tilde{m})mm^* = 0$ to obtain $(1- m\tilde{m})m = 0$.
After including the polynomial associated to this new identity in the set of translated assumptions, the software manages to verify the ideal membership of all polynomials corresponding to the claimed identities fully automatically, and thereby, proves the claimed statement.

The proof of $(i) \Rightarrow (v)$ of Theorem~\ref{t3} using the software essentially proceeds along the same lines as the proof discussed above concerning the same implication in Hartwig's theorem.
The only difference is that now also the right $*$-cancellability of $m$ has to be shown.
To this end, we include the identity $zmm^*= 0$ in the assumptions and prove $zm = 0$ with an arbitrary ring element $z$.
When translating these identities into polynomials, $z$ has to be replaced by a new indeterminate that does not satisfy any additional identities.
The software then proves the ideal membership of the polynomial associated to the claimed identity in the ideal generated by the polynomials representing the assumptions fully automatically. 

\begin{remark}
We note that in a similar fashion to the implications discussed above, also all other implications of Theorem~\ref{t3} and all other results presented in this work, including Theorems~\ref{t0}, \ref{t1}, \ref{t3}, \ref{t4}, \ref{tnova}, \ref{t5}, and \ref{t00}, can be proven using the framework.
The relevant computations with noncommutative polynomials were done using \texttt{OperatorGB} and are available at \url{http://gregensburger.com/softw/OperatorGB} along with a file containing all the certificates of ideal membership.
Since all cofactor representations obtained have only polynomials with integer coefficients, by applying Theorem~\ref{thm:RlinearUnique}, the corresponding theorems hold for any setting in which they can be formulated like rings with involution, (rectangular) matrices over such rings, and linear bounded operators between Hilbert spaces.
\end{remark}

\section*{Acknowledgements}

We thank Anja Korporal, Marko Petkovi\'c, and Milan Tasi\'c for discussions related to this paper in the course of the OeAD project SRB 05/2016.  
This work was supported by the Ministry of Science, Technology and Development, Republic of Serbia, and by the Austrian Science Fund (FWF): P~27229, P~31952, and P~32301.

\appendix
\section{Formal summary of algebraic proof framework}
\label{sec:formalsummary}

Now, we give a more formal explanation of the framework developed in \cite{RaabRegensburgerHosseinPoor2019}.
In the following, we fix a set $X$ and a commutative ring $R$ with unit element. We consider the ring $R\langle{X}\rangle$ of noncommutative polynomials with coefficients in $R$ and indeterminates in $X$, where indeterminates do not commute with each other but with coefficients.

Recall that a \emph{quiver} is given by a tuple $(V,E,s,t)$ where $V$ is the set of vertices, $E$ is the set of edges, and $s,t : E\to{V}$ give the \emph{source} $s(e)$ and \emph{target} $t(e)$ of each edge $e \in E$.
We consider \emph{labelled quivers} where edges have labels in $X$, i.e.\ with a function $l:E\to{X}$ giving the labels of edges.
In the following, we fix a labelled quiver $Q=(V,E,X,s,t,l)$ such that edges have unique labels, i.e.\ $l$ is injective.
Based on the labels of edges, it is straightforward to label paths in $Q$ so that multiplication of labels as monomials corresponds to concatenation of paths.
Likewise, the notion of source and target of edges can be naturally extended to paths.

A polynomial in $R\langle{X}\rangle$ such that all its monomials are labels of paths in $Q$ that have the same source and the same target is called \emph{compatible} with $Q$. 
For vertices $v,w \in V$, we collect all compatible polynomials arising from paths with source $v$ and target $w$ in the set $R\langle{X}\rangle_{v,w}$, which is an $R$-module.
Note that for the case $v=w$ there exists an empty path from $v$ to $w$, which has the constant monomial $1$ as its label.
By construction, the polynomials $f_1,\dots,f_8,f$ defined in Section~\ref{sec:introductionframework} are compatible with the following labelled quiver.
\begin{figure}[!ht]
\center
  \begin{tikzpicture}
    \matrix (m) [matrix of math nodes, column sep=2cm]
     {\bullet & \bullet & \bullet\\};
    \path[->] (m-1-1) edge [bend left] node [auto] {$a^{-}$} (m-1-2);
    \path[->] (m-1-2) edge [bend left] node [auto] {$a$} (m-1-1);
    \path[->] (m-1-2) edge [loop above] node [auto] {$i$} (m-1-2);
    \path[->] (m-1-2) edge [bend left] node [auto] {$b^{-}$} (m-1-3);
    \path[->] (m-1-3) edge [bend left] node [auto] {$b$} (m-1-2);
  \end{tikzpicture}
\caption{Labelled quiver for Werner's theorem}
\label{fig:WernerQuiver}
\end{figure}
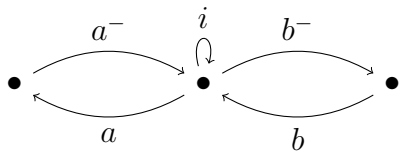

A \emph{representation} of a quiver $(V,E,s,t)$ can be specified by a pair $(\mathcal{M},\varphi)$ such that $\mathcal{M}=(\mathcal{M}_v)_{v\in V}$ is a family of $R$-modules and the map $\varphi$ assigns to each $e\in E$ an $R$-linear map $\varphi(e):\mathcal{M}_{s(e)}\to\mathcal{M}_{t(e)}$.
For example, with $R=\mathbb{Z}$ or $R=\mathbb{C}$, the two diagrams in Section~\ref{sec:introductionframework} specify representations of the labelled quiver shown in Figure~\ref{fig:WernerQuiver}.

Now, for a given representation $(\mathcal{M},\varphi)$ of $Q$, plugging in the $R$-linear maps $\varphi(e)$, $e \in E$, for the indeterminates $l(e)$ of polynomials in $R\langle{X}\rangle$ can be formalized as follows.
For every nonconstant monomial $m \in R\langle{X}\rangle_{v,w}$, there exists a nonempty path $e_n{\dots}e_1$ in $Q$ with source $v$, target $w$, and label $m$, which allows to define the $R$-linear map $\varphi_{v,w}(m):=\varphi(e_n){\cdot}{\dots}{\cdot}\varphi(e_1)$ from $\mathcal{M}_v$ to $\mathcal{M}_w$.
Note that, by definition of $\varphi$, the composition of the maps $\varphi(e_i)$ exists.
Similarly, if $v=w$, we define $\varphi_{v,v}(1) := \id_{\mathcal{M}_v}$.
The map $\varphi_{v,w}$ extends $R$-linearly to all $f\in R\langle{X}\rangle_{v,w}$ and we call the $R$-linear map $\varphi_{v,w}(f)$ a \emph{realization} of $f$ w.r.t.\ the representation $(\mathcal{M},\varphi)$ of $Q$.

Altogether, one can prove the following main theorem about the framework. The formulation stated here is a consequence of Theorem~32 and 15 in \cite{RaabRegensburgerHosseinPoor2019}.

\begin{theorem}\label{thm:RlinearUnique}
 Let $R$ be a commutative ring with unit element, let $F \subseteq R\langle{X}\rangle$ be a set of polynomials without a constant term, and let $f \in (F)$.
 Then, for all labelled quivers $Q$ with unique labels in $X$ such that $f$ and all polynomials in $F$ are compatible with $Q$ and for all representations $(\mathcal{M},\varphi)$ of $Q$ such that the realizations of the polynomials in $F$ w.r.t.\ $(\mathcal{M},\varphi)$ are zero, we have that also the realization of $f$ w.r.t.\ $(\mathcal{M},\varphi)$ is zero.
\end{theorem}

 All notions and results of this section naturally generalize to $R$-linear categories by considering objects and morphisms in such a category instead of $R$-modules and $R$-linear maps, respectively.
 For more details, see Section~5.2 in \cite{RaabRegensburgerHosseinPoor2019}.
Based on a refined version of the framework using rewriting, it is possible to obtain a similar theorem where polynomials in $F$ are allowed to have a constant term, see Theorem~32 in \cite{ChenavierHofstadlerRaabRegensburger2020}.

Altogether, based on the theorem above, we obtain a rigorous proof of the following statement for matrices discussed in Section~\ref{sec:introductionframework}.
\begin{lemma}
\label{lem:Werner1}
 Let $A,B$ be matrices with entries in a commutative ring $R$ with unit element and let $A^-,B^-$ be inner inverses of $A$ resp.\ $B$. If $BB^-(I-A^-A)=I-A^-A$ holds, then $B^-A^-$ is an inner inverse of $AB$.
\end{lemma}
\begin{proof}
In the polynomial ring $R\langle{a,a^-,b,b^-,i}\rangle$, the cofactor representation \eqref{eq:WernerCofactor1} shows that the polynomial $f$ given by \eqref{eq:WernerClaimPoly} lies in the ideal $(F) \subseteq R\langle{a,a^-,b,b^-,i}\rangle$, where $F:=\{f_1,\dots,f_8\}$. The generators of the ideal as well as the polynomial $f$ are compatible with the labelled quiver shown in Figure~\ref{fig:WernerQuiver}. We fix the following representation of this quiver.
\begin{center}
  \begin{tikzpicture}
    \matrix (m) [matrix of math nodes, column sep=2cm]
     {R^m & R^n & R^k\\};
    \path[->] (m-1-1) edge [bend left] node [auto] {$A^{-}$} (m-1-2);
    \path[->] (m-1-2) edge [bend left] node [auto] {$A$} (m-1-1);
    \path[->] (m-1-2) edge [loop above] node [auto] {$I$} (m-1-2);
    \path[->] (m-1-2) edge [bend left] node [auto] {$B^{-}$} (m-1-3);
    \path[->] (m-1-3) edge [bend left] node [auto] {$B$} (m-1-2);
  \end{tikzpicture}
\end{center}
If $BB^-(I-A^-A)=I-A^-A$, then the realizations of all elements of $F$ are zero by assumption. Then, the realization of $f$ is zero by Theorem~\ref{thm:RlinearUnique}, i.e.
\[
 ABB^-A^-AB-AB=0.\qedhere
\]
\end{proof}

Note that the proof of this lemma relies on the purely algebraic fact that the polynomial $f$ representing the claim lies in the ideal $(F)$ representing the assumptions.
By changing the representation of the quiver, Theorem~\ref{thm:RlinearUnique} gives rigorous proofs also of analogous lemmas for bounded linear operators between Hilbert spaces, for homomorphisms of $R$-modules, and for ring elements.
\end{document}